\documentclass[11pt,reqno]{amsart}
\usepackage{amssymb,latexsym}
\usepackage{graphicx}
\usepackage{amssymb}
\usepackage{url}		

\newcommand{\R}{{\mathbb R}}

\newcommand{\N}{{\mathbb N}}
\newcommand{\Z}{{\mathbb Z}}

\theoremstyle{plain}
\numberwithin{equation}{section}
\newtheorem{thm}{Theorem}[section]

\newtheorem{lemma}[thm]{Lemma}

\newtheorem{definition}[thm]{Definition}
\newtheorem{proposition}[thm]{Proposition}
\newtheorem{corollary}[thm]{Corollary}

\usepackage{mathrsfs}
\usepackage{diagbox}

\allowdisplaybreaks

\begin{document}

\setcounter{page}{1}

\title[Product Representation of Fibonacci Numbers]{On the Product Representation of Number Sequences, with Application to the Fibonacci Family}
\author{Michelle Rudolph-Lilith}
\address{Unit\'e de Neurosciences, Information et Complexit\'e, CNRS\\
		1 Ave de la Terrasse\\
		91198 Gif-sur-Yvette\\
		France}
\email{rudolph@unic.cnrs-gif.fr}
\thanks{Research supported in part by CNRS. The author wishes to thank LE Muller II and OD Little for valuable comments.}


\begin{abstract}
We investigate general properties of number sequences which allow explicit representation in terms of products. We find that such sequences form whole families of number sequences sharing similar recursive identities. Restricting to the cosine of fractional angles, we then study the special case of the family of $k$-generalized Fibonacci numbers, and present general recursions and identities which link these sequences.
\end{abstract}

\maketitle


\section{Introduction}

It has long been known that Fibonacci and Pell numbers, defined by
\begin{equation}
\label{Eq_Fibs}
F_0 = 0, F_1 = 1, F_n = F_{n-1} + F_{n-2}
\end{equation}
and
\begin{equation}
\label{Eq_Pells}
P_0 = 0, P_1 = 1, P_n = 2 P_{n-1} + P_{n-2}
\end{equation}
with $n \geq 2$, respectively, can be represented in product form (e.g., see \cite{Rutherford47, Rutherford52, Lind65, Zeitlin67, Shapiro94}), specifically
\begin{equation}
F_n 
= \prod\limits_{l=1}^{\lfloor \frac{n-1}{2} \rfloor} \left( 3 + 2 \cos \left[ \frac{2 l \pi}{n} \right] \right)
\equiv \prod\limits_{l=1}^{n-1} \left( 1 - 2i \cos \left[ \frac{l \pi}{n} \right] \right)
\end{equation}
and
\begin{equation}
P_n 
= 2^{\lfloor \frac{n}{2} \rfloor} \prod\limits_{l=1}^{\lfloor \frac{n-1}{2} \rfloor} \left( 3 + \cos \left[ \frac{2 l \pi}{n} \right] \right)
\equiv \prod\limits_{l=1}^{n-1} \left( 2 - 2i \cos \left[ \frac{l \pi}{n} \right] \right)
\end{equation}
$\forall n \in \N, n \geq 2$. In fact, already in the original solution to Problem H-64 \cite{Lind65}, Zeitlin showed that
\begin{equation}
\label{Eq_ProdRepLucas}
L_n^{(p,q)} = q^{\frac{n-1}{2}} \prod\limits_{l=1}^{n-1} \left( \frac{p}{\sqrt{q}} - 2 \cos \left[ \frac{l \pi}{n} \right] \right)
\end{equation}
$p,q \in \R$, provides a valid product representation of all members in general Lucas sequences, the latter being defined by the recursive relation
\begin{equation}
L_0^{(p,q)} = 0, L_1^{(p,q)} = 1, L_n^{(p,q)} = p L_{n-1}^{(p,q)} - q L_{n-2}^{(p,q)}
\end{equation}
for $n \geq 2$ \cite{Zeitlin67}. Later, expressions of the form (\ref{Eq_ProdRepLucas}) were used to obtain other explicit representations of the corresponding number sequences in terms of finite power series in the sequence parameters (e.g., see \cite{AndreJeannin94, HendelCook96, Cigler03, Sun06}), thus highlighting the importance and usefulness of such product representations for the investigation of number sequences.

In this contribution, we will show that product representations of number sequences can also be utilised to establish direct links between different sequences. To that end, we formulate

\begin{definition} 
\label{Def_Family}
{\bf (Family of Number Sequences)} Let $\{ x_{n,l} \}$ with $x_{n,l} \in \R$ and $n, l \in \N$ be an arbitrary two-parameter set of numbers. The corresponding family of number sequences $\{ X_{n,m} \}$ is defined by the set of all $X_{n,m}$ with
\begin{equation}
\label{Eq_Xnm}
X_{n,m} = \prod\limits_{l=1}^n ( m + x_{n,l} ) \, ,
\end{equation}
where $m \in \Z$ labels the individual number sequences within the family, and $n$ the members of each sequence. 
\end{definition}

A concrete example of such a family is given if we set $x_{n,l} = - 2 i \cos[ \tfrac{l \pi}{n+1} ]$. In this case, using (\ref{Eq_ProdRepLucas}) with $q = -1$ and $p \in \Z$, we have
\begin{equation}
\label{Eq_GenFibProd}
X_{n,p} = \prod\limits_{l=1}^n \left( p - 2i\cos[ \tfrac{l \pi}{n+1} ] \right) \equiv L_{n+1}^{(p,-1)} \, ,
\end{equation}
thus $X_{n,p}$ defines the family of generalized Fibonacci sequences $F_n^{(p)} := L_n^{(p,-1)}, n \geq 2$ obeying the recursive relation
\begin{equation}
\label{Eq_GenFibRec}
F_0^{(p)} = 0, F_1^{(p)} = 1, F_n^{(p)} = p F_{n-1}^{(p)} + F_{n-2}^{(p)} \, .
\end{equation}
We will call this family the Fibonacci family, and explore some of its properties in Section \ref{S_FibFamily}.

The paper is organized as follows. In Section \ref{S_ProdRep}, we will prove various general properties of the number sequences $X_{n,m}$ within a given family, with focus on linear recursive relations between the individual sequences. Two simple examples will be investigated in Section \ref{S_Examples}, and Section \ref{S_FibFamily} will focus on the Fibonacci family defined in (\ref{Eq_GenFibProd}). Some generalizations will be discussed at the end.


\section{Product Representation of Certain Number Sequences}
\label{S_ProdRep}

For any given set of numbers $\{ x_{n,l} \}$, we define
\begin{equation}
\label{Eq_calX0}
\mathscr{X}_n := \sum\limits_{l=1}^n x_{n,l} \, .
\end{equation}
We will first express $\mathscr{X}_n$ in terms of the associated number sequences $X_{n,m}$ defined in (\ref{Eq_Xnm}):

\begin{lemma}
\label{L_calX}
For any given set of numbers $\{ x_{n,l} \}$ with $x_{n,l} \in \R, n,l \in \N$, the sum over $x_{n,l}$ is given by
\begin{equation}
\label{Eq_calX}
\mathscr{X}_n = \frac{(-1)^n}{n!} \sum\limits_{l=1}^n (-1)^l \binom{n}{l} \, l \, X_{n,l} - \frac{1}{2} n (n+1) \, ,
\end{equation}
where $X_{n,m}, m \in \Z$ denotes the number sequences associated with $\{ x_{n,l} \}$.
\end{lemma}

\begin{proof}
We first construct a system of equations by explicitly factorizing (\ref{Eq_Xnm}) for successive $m \in [1,n]$. To that end, we define
\begin{equation*}
\mathcal{X}_n^{(p)} := \sum\limits_{\substack{l_i=1 \\ i \in [1,p] \\ l_{j+1}>l_j \, \forall l_j}}^n \prod\limits_{i=1}^p x_{n,l_i}
\end{equation*}
and obtain from (\ref{Eq_Xnm}) a system of linear equations in $\mathcal{X}_n^{(p)}, p \in [1,n]$:
\begin{eqnarray*}
X_{n,1} & = & 1^n + 1^{(n-1)} \mathcal{X}_n^{(1)} + 1^{(n-1)} \mathcal{X}_n^{(2)} + \ldots + \mathcal{X}_n^{(n)} \\
X_{n,2} & = & 2^n + 2^{(n-1)} \mathcal{X}_n^{(1)} + 2^{(n-2)} \mathcal{X}_n^{(2)} + \ldots + \mathcal{X}_n^{(n)} \\
X_{n,3} & = & 3^n + 3^{(n-1)} \mathcal{X}_n^{(1)} + 3^{(n-2)} \mathcal{X}_n^{(2)} + \ldots + \mathcal{X}_n^{(n)} \\
        & \vdots & \\
X_{n,n} & = & n^n + n^{(n-1)} \mathcal{X}_n^{(1)} + n^{(n-2)} \mathcal{X}_n^{(2)} + \ldots + \mathcal{X}_n^{(n)} \, .
\end{eqnarray*}
This system can be written in more compact form as
\begin{equation}
\label{Eq_L1p1}
X_{n,i} - i^n = \sum\limits_{j=1}^{n} a_{ij} \mathcal{X}_n^{(j)} \, ,
\end{equation}
where $a_{ij} = i^{n-j}$, $i,j \in [1,n]$. What remains is to solve (\ref{Eq_L1p1}) for $\mathcal{X}_n^{(1)} \equiv \mathscr{X}_n$. To that end, we note that $a_{in} = 1, \forall i \in [1,n]$, which allows us to construct a new system of $n-1$ equations by subtracting successive equations in (\ref{Eq_L1p1}). We obtain
\begin{equation}
\label{Eq_L1p2}
X_{n,i+1} - X_{n,i} - \left( (i+1)^n - i^n \right) = \sum\limits_{j=1}^{n-1} a^{(1)}_{ij} \mathcal{X}_n^{(j)} \, ,
\end{equation}
where 
\begin{equation*}
a^{(1)}_{ij} = a_{i+1,j} - a_{ij} = \left( (i+1)^{n-j} - i^{n-j} \right) \equiv \sum_{l=0}^1 (-1)^{l+1} \binom{1}{l} (i+l)^{n-j}
\end{equation*}
for $i,j \in [1,n-1]$. Again, $a^{(1)}_{i,n-1} = 1, \forall i \in [1,n-1]$, and we can further reduce the system (\ref{Eq_L1p2}) by subtracting successive equations. After $m$ repetitions, we have
\begin{equation}
(-1)^m \sum\limits_{l=0}^m (-1)^l \binom{m}{l} X_{n,i+l} - (-1)^m \sum\limits_{l=0}^m (-1)^l \binom{m}{l} (i+l)^n
= \sum\limits_{j=1}^{n-m} a_{ij}^{(m)} \mathcal{X}_n^{(j)}
\end{equation}
for $i,j \in [1,n-m]$, where 
\begin{equation*}
a_{ij}^{(m)} = (-1)^m \sum\limits_{l=0}^m (-1)^l \binom{m}{l} (i+l)^{n-j}
\end{equation*}
with $a_{i,n-m}^{(m)} = m!$ for all $i \in [1,n-m]$. For $m=n-1$, we finally obtain
\begin{equation*}
(-1)^{n-1} \sum\limits_{l=0}^{n-1} (-1)^l \binom{n-1}{l} X_{n,1+l} - (-1)^{n-1} \sum\limits_{l=0}^{n-1} (-1)^l \binom{n-1}{l} (1+l)^n
= (n-1)! \, \mathcal{X}_n^{(1)} \, .
\end{equation*}
Changing the summation variable $l \rightarrow l+1$ and observing that 
\begin{equation*}
\frac{(-1)^n}{n!} \sum\limits_{l=1}^n (-1)^l \binom{n}{l} l^{n+1} = -\frac{1}{2} n (n+1)
\end{equation*}
(Gould (1.14), \cite{Gould72}),  we finally arrive at (\ref{Eq_calX}).
\end{proof}

Lemma \ref{L_calX} provides, for any given $n \geq 1$, an explicit representation of the sum over $x_{n,l}$, equation (\ref{Eq_calX}),  in terms of a finite linear combination of the number sequences $X_{n,m}$. With this, we can immediately formulate

\begin{lemma}
\label{L_calXm}
For any given set of numbers $\{ x_{n,l} \}$ with $x_{n,l} \in \R, n,l \in \N$ and associated family of number sequences $\{ X_{n,m} \}, m \in \Z$, the sum over $x_{n,l}$ obeys the identities
\begin{equation}
\label{Eq_calXm1}
\mathscr{X}_n = \frac{(-1)^n}{n!} \sum\limits_{l=1}^n (-1)^l \binom{n}{l} \, l \, X_{n,l+m} - \frac{1}{2} n (n+1) - n m 
\end{equation}
and, for $m \neq 0$,
\begin{equation}
\label{Eq_calXm2}
\mathscr{X}_n = \frac{(-1)^n}{n! \, m^{n-1}} \sum\limits_{l=1}^n (-1)^l \binom{n}{l} \, l \, X_{n,l m} - \frac{1}{2} n (n+1) m \, .  
\end{equation}
\end{lemma}

\begin{proof}
We first prove (\ref{Eq_calXm1}). Let us define $X_{n,m}$ and $\mathscr{X}_n$ for the set of numbers $(m' + x_{n,l})$:
\begin{eqnarray*}
X_{n,m}^{(m')} := \prod\limits_{l=1}^n \left( m + ( m' + x_{n,l} ) \right) \\
\mathscr{X}_n^{(m')} := \sum\limits_{l=1}^n ( m' + x_{n,l} )
\end{eqnarray*}
for arbitrary $m' \in \Z$. From the first equation and definition (\ref{Eq_Xnm}), it follows immediately that $X_{n,m}^{(m')} = X_{n,m+m'}$ and $\mathscr{X}_n^{(m')} = n m' + \mathscr{X}_n$, which together with (\ref{Eq_calX}) yield (\ref{Eq_calXm1}).

The second relation (\ref{Eq_calXm2}) can be shown in a similar fashion. We define $X_{n,m}$ and $\mathscr{X}_n$ for the set of numbers $x_{n,l} / m', m' \in \Z, m \neq 0$:
\begin{eqnarray*}
\tilde{X}_{n,m}^{(m')} := \prod\limits_{l=1}^n \left( m + \frac{x_{n,l}}{m'} \right) \\
\tilde{\mathscr{X}}_n^{(m')} := \sum\limits_{l=1}^n \frac{x_{n,l}}{m'} \, ,
\end{eqnarray*}
from which follows that 
\begin{equation}
\tilde{X}_{n,m}^{(m')} = \frac{1}{m'^n} \prod\limits_{l=1}^n ( m m' + x_{n,l} ) = \frac{1}{m'^n} X_{n,mm'}
\end{equation}
and $\tilde{\mathscr{X}}_n^{(m')} = \mathscr{X}_n / m'$. Using again (\ref{Eq_calX}), we obtain (\ref{Eq_calXm2}).
\end{proof} 

Lemma \ref{L_calXm} is interesting in various respects. It not just generalizes (\ref{Eq_calX}), but also shows that various combinations of $X_{n,m}$ within a given family of number sequences must yield the same result $\mathscr{X}_n$. This, in turn, allows us to construct general relations between $X_{n,m}$, which will hold for all families of number sequences $\{ X_{n,m} \}$ representable in product form (\ref{Eq_Xnm}), and constitute the main result of this contribution. We can formulate

\begin{proposition}
\label{P_XnmIdent}
The members $X_{n,m}$ of a given family of number sequences $\{ X_{n,m} \}, m \in Z$ and $n \in \N$, obey the general recursive relation
\begin{equation}
\label{Eq_XnmRec}
X_{n,m+1} = (-1)^n \sum\limits_{l=1}^n (-1)^l \binom{n}{l-1} X_{n,l+m-n} + n!
\end{equation}
and are subject to the identity
\begin{equation}
\label{Eq_XnmNeg1}
\frac{1}{m^{n-1}} \sum\limits_{l=1}^n (-1)^l \binom{n}{l} \, l \, X_{n,lm}
= \sum\limits_{l=1}^n (-1)^l \binom{n}{l} \, l \, X_{n,l} + \frac{(-1)^{n-1}}{2} (1-m)n(n+1)!
\end{equation}
for $m \neq 0$.
\end{proposition}

\begin{proof}
The proof of (\ref{Eq_XnmRec}) utilizes (\ref{Eq_calXm1}) for $m \rightarrow m-n+1$ and $m \rightarrow m-n$, yielding
\begin{eqnarray*}
\mathscr{X}_n 
& = & \frac{(-1)^n}{n!} \sum\limits_{l=1}^n (-1)^l \binom{n}{l} \, l \, X_{n,l+m-n+1} - \frac{1}{2} n (n+1) - n (m-n+1) \\ 
& = & \frac{(-1)^n}{n!} \sum\limits_{l=1}^{n-1} (-1)^l \binom{n}{l} \, l \, X_{n,l+m-n+1} + \frac{1}{(n-1)!} X_{n,m+1} - \frac{1}{2} n (n+1) - n (m-n+1)
\end{eqnarray*}
and
\begin{eqnarray*}
\mathscr{X}_n 
& = & \frac{(-1)^n}{n!} \sum\limits_{l=1}^n (-1)^l \binom{n}{l} \, l \, X_{n,l+m-n} - \frac{1}{2} n (n+1) - n (m-n) \\
& = & \frac{(-1)^n}{n!} \sum\limits_{l=1}^{n-1} (-1)^{l+1} \binom{n}{l+1} \, (l+1) \, X_{n,l+m-n+1} + \frac{(-1)^{n+1}}{(n-1)!} X_{n,m-n+1} \\
&   & - \frac{1}{2} n (n+1) - n (m-n) \, ,
\end{eqnarray*}
respectively, where in the last step $l \rightarrow l-1$ was used. Subtracting both identities and observing that $(l+1) \binom{n}{l+1} + l \binom{n}{l} = n \binom{n}{l}$, we obtain
\begin{eqnarray*}
0 
& = & \frac{(-1)^{n+1}}{(n-1)!} \sum\limits_{l=1}^{n-1} (-1)^l \binom{n}{l} X_{n,l+m-n+1} + \frac{(-1)^{n+1}}{(n-1)!} X_{n,m-n+1} - \frac{1}{(n-1)!} X_{n,m+1} + n \\
& = & \frac{(-1)^{n+1}}{(n-1)!} \sum\limits_{l=0}^{n-1} (-1)^l \binom{n}{l} X_{n,l+m-n+1} - \frac{1}{(n-1)!} X_{n,m+1} + n,
\end{eqnarray*}
which, after a change of the summation variable $l \rightarrow l+1$, yields (\ref{Eq_XnmRec}).

In a similar fashion, (\ref{Eq_XnmNeg1}) is a direct consequence of subtracting equations (\ref{Eq_calXm1}) and (\ref{Eq_calXm2}) in Lemma \ref{L_calXm}.
\end{proof}

Equation (\ref{Eq_XnmRec}) in Proposition \ref{P_XnmIdent} provides general linear recursions in $m \in \Z$ for $X_{n,m}$. The form of these recursions depends on $n$, and contains an increasing number of terms for increasing $n$. Specifically, for any given $n$, (\ref{Eq_XnmRec}) expresses $X_{n,m}$ in terms of $X_{n,m'}$ with $m' \in [m-n,m-1]$. Based on these recursions, using the generating function approach, we can deduce explicit identities which express $X_{n,m}$ and $X_{n,-m}$ for $m \geq n$, in terms of $X_{n,m'}$ with $m' \in [0,n-1]$ and $m' \in [-n+1,0]$, respectively:

\begin{corollary}
\label{Cor_Xnm1}
For any given family of number sequences $\{ X_{n,m} \}$, the following identities hold
\begin{eqnarray}
X_{n,m} & = & \sum\limits_{l=0}^{n-1} (-1)^{n+l} \frac{n-l}{l-m} \binom{m}{n} \binom{n}{l} X_{n,l} + \frac{m!}{(m-n)!} \label{Eq_XnmExp1} \\
X_{n,-m} & = & \sum\limits_{l=0}^{n-1} (-1)^{n+l} \frac{n-l}{l-m} \binom{m}{n} \binom{n}{l} X_{n,-l} + (-1)^n \frac{m!}{(m-n)!}\label{Eq_XnmExp2}
\end{eqnarray}
for all $n \in \N$ and $m \in \Z$ with $m \geq n$.
\end{corollary}

\begin{proof}
We start with (\ref{Eq_XnmRec}) for $m+1 \rightarrow m$
\begin{equation*}
X_{n,m} = \sum\limits_{l=1}^n (-1)^{n+l} \binom{n}{l-1} X_{n,l+m-n-1} + n!
\end{equation*}
and define the general generating function
\begin{equation*}
A[z] := \sum\limits_{m \geq 0} X_{n,m} z^m
\end{equation*}
for arbitrary $z \in \R, z \neq 0$. Multiplication of $X_{n,m}$ with $z^m$ and summation over $m \geq n$ yields after lengthy yet straightforward manipulations
\begin{eqnarray*}
A[z] 
& = & \sum\limits_{m=0}^{n-1} X_{n,m} z^m + \sum\limits_{m \geq n} \sum\limits_{l=0}^{n-1} \binom{m-l+n-1}{n-1} X_{n,l} z^m \\
&   & - \sum\limits_{m \geq n} \sum\limits_{k=1}^n \sum\limits_{l=0}^{k-2} (-1)^{n+k} \binom{n}{k-1} \binom{m-l+k-2}{n-1} X_{n,l} z^m + \sum\limits_{m \geq n} \binom{m}{n} n! z^m \\
& = & \sum\limits_{m=0}^{n-1} X_{n,m} z^m + \sum\limits_{m \geq n} \sum\limits_{l=0}^{n-1} \binom{m-l+n-1}{n-1} X_{n,l} z^m \\
&   & - \sum\limits_{m \geq n} z^m \sum\limits_{l=0}^{n-2} X_{n,l} \sum\limits_{k=l}^{n-2} (-1)^{n+k} \binom{n}{k+1} \binom{m+k-l}{n-1} + \sum\limits_{m \geq n} \binom{m}{n} n! z^m \, ,
\end{eqnarray*}  
where in the last step, for each $m$, we appropriately reordered terms occurring in the third sum. Collecting coefficients for any given $m \geq n$ and observing that 
\begin{equation*}
\sum\limits_{k=l}^{n-2} (-1)^{n+k} \binom{n}{k+1} \binom{m+k-l}{n-1} 
= \binom{m+n-l-1}{n-1} - (-1)^{n+l} \frac{n-l}{l-m} \binom{m}{n} \binom{n}{l} \, ,
\end{equation*}
we obtain (\ref{Eq_XnmExp1}). Equation (\ref{Eq_XnmExp2}) can be shown in a similar fashion.
\end{proof}

Finally, equation (\ref{Eq_XnmRec}) also allows to deduce a number of general identities the number sequences $X_{n,m}$ of any given family must obey. Specifically, we have

\begin{corollary}
\label{Cor_Xnm2}
For $n,p,q \in \N, n \geq p + 1$ with $p \geq 1$, $0 \leq q < p$ and $m \in \Z$, the sequences $X_{n,m}$ of any given family of number sequences $\{ X_{n,m} \}$ are subject to the following identities:
\begin{eqnarray}
0 & = & \sum\limits_{l=0}^n (-1)^l \binom{n}{l} \, l^q \, X_{n-p,m-n+l} \label{Eq_XnmId1} \\
0 & = & \sum\limits_{l=0}^n (-1)^l \binom{n}{l} \, l^p \, X_{n-p,m-n+l} - (-1)^n n! \, .\label{Eq_XnmId2}
\end{eqnarray}
\end{corollary}

\begin{proof}
We first show that (\ref{Eq_XnmId1}) is valid for the special case $q=0$, i.e. 
\begin{equation*} 
0 = \sum\limits_{l=0}^n (-1)^l \binom{n}{l} X_{n-p,m-n+l} \, .
\end{equation*}
This can be shown by induction in $p$. From (\ref{Eq_XnmRec}), considering $X_{n,m}$, after change of the summation variable $l \rightarrow l+1$ and observing that $(-1)^{2n+1} \binom{n}{n} \equiv -1$, we have
\begin{equation*}
0 = (-1)^{n+1} \sum\limits_{l=0}^n (-1)^l \binom{n}{l} X_{n,m-n+l} + n! \, ,
\end{equation*}
which yields for $n \rightarrow n-1 \geq 1$ and arbitrary $m$ 
\begin{equation*}
0 = (-1)^{n} \sum\limits_{l=0}^{n-1} (-1)^l \binom{n-1}{l} X_{n-1,m-n+l+1} + (n-1)! \, ,
\end{equation*}
and for $m \rightarrow m-1$
\begin{equation*}
0 = (-1)^{n} \sum\limits_{l=0}^{n-1} (-1)^l \binom{n-1}{l} X_{n-1,m-n+l} + (n-1)! \, .
\end{equation*}
Subtracting the last two identities yields
\begin{eqnarray*}
0 
& = & (-1)^{n} \sum\limits_{l=0}^{n-2} (-1)^l \binom{n-1}{l} X_{n-1,m-n+l+1} - X_{n-1,m} \\
&   & - (-1)^n X_{n-1,m-n} - (-1)^{n} \sum\limits_{l=1}^{n-1} (-1)^l \binom{n-1}{l} X_{n-1,m-n+l} \\
& = & - (-1)^n X_{n-1,m-n} + (-1)^{n+1} \sum\limits_{l=1}^{n-1} (-1)^l \binom{n}{l} X_{n-1,m-n+l} - X_{n-1,m} \\
& \equiv & (-1)^{n+1} \sum\limits_{l=0}^{n} (-1)^l \binom{n}{l} X_{n-1,m-n+l} \, ,
\end{eqnarray*}
where in the first step the summation variable of the first sum was changed according to $l \rightarrow l-1$, and $\binom{n-1}{l-1} + \binom{n-1}{l} = \binom{n}{l}$ was used. This proves (\ref{Eq_XnmId1}) for the special case $q=0$ for $p=1$. Assuming now (\ref{Eq_XnmId1}) is true for $q=0$ and a given $p \geq 1$, we have for $n \rightarrow n-1$
\begin{equation*}
0 = (-1)^{n} \sum\limits_{l=0}^{n-1} (-1)^l \binom{n-1}{l} X_{n-1-p,m-(n-1)+l}
\equiv (-1)^{n} \sum\limits_{l=0}^{n-1} (-1)^l \binom{n-1}{l} X_{n-(p+1),m+1-n+l} \, ,
\end{equation*}
which yields for $m \rightarrow m-1$
\begin{equation*}
0 = (-1)^{n} \sum\limits_{l=0}^{n-1} (-1)^l \binom{n-1}{l} X_{n-(p+1),m-n+l} \, .
\end{equation*}
Subtracting the last two equations, we obtain after manipulations similar to the $p=1$ case above
\begin{eqnarray*}
0 
& = & - (-1)^n X_{n-(p+1),m-n} + (-1)^{n+1} \sum\limits_{l=1}^{n-1} (-1)^l \binom{n}{l} X_{n-(p+1),m-n+l} - X_{n-(p+1),m} \\
& \equiv & (-1)^{n+1} \sum\limits_{l=0}^{n} (-1)^l \binom{n}{l} X_{n-(p+1),m-n+l} \, ,
\end{eqnarray*}
thus proving (\ref{Eq_XnmId1}) for $q=0$ and $p+1$.

Identity (\ref{Eq_XnmId1}) for $0 < q < p$ can be shown by induction in $q$. Above we demonstrated the validity of (\ref{Eq_XnmId1}) for $q=0$ and all $p \geq 1$. Assuming now that (\ref{Eq_XnmId1}) is true for a given $q \geq 0$ and all $p \geq q+1$, we obtain for $n \rightarrow n-1$ and $m \rightarrow m-1$
\begin{eqnarray*}
0 
& = & \sum\limits_{l=0}^{n-1} (-1)^l \binom{n-1}{l} \, l^q \, X_{n-p,m-n+l} \\
& = & \sum\limits_{l=0}^{n-1} (-1)^l \binom{n}{l} \, \frac{n-l}{n} \, l^q \, X_{n-p,m-n+l} \\
& = & \sum\limits_{l=0}^{n-1} (-1)^l \binom{n}{l} \, l^q \, X_{n-p,m-n+l} 
      - \frac{1}{n} \sum\limits_{l=0}^{n-1} (-1)^l \binom{n}{l} \, l^{q+1} \, X_{n-p,m-n+l} \\
& = & -(-1)^n n^q X_{n-p,m} - \frac{1}{n} \sum\limits_{l=0}^{n-1} (-1)^l \binom{n}{l} \, l^{q+1} \, X_{n-p,m-n+l} \\
& \equiv & \sum\limits_{l=0}^{n} (-1)^l \binom{n}{l} \, l^{q+1} \, X_{n-p,m-n+l} \, ,
\end{eqnarray*}
where the binomial identity $\binom{n}{l} = \frac{n}{n-l} \binom{n-1}{l}$ was utilized, thus proving (\ref{Eq_XnmId1}) for $q+1 \leq p$, i.e. $q < p$.

Identity (\ref{Eq_XnmId2}) can be shown in an equivalent fashion through induction in $p$, utilizing (\ref{Eq_XnmRec}) and (\ref{Eq_XnmId1}).
\end{proof}

We finally note that the recursive relation (\ref{Eq_XnmRec}) and identities listed in Corollaries \ref{Cor_Xnm1} and \ref{Cor_Xnm2} are general and hold for each family of number sequences $\{ X_{n,m} \}$, thus suggesting that all families constructed from number sequences of the form (\ref{Eq_calX}) are governed by identical relationships between their sequences. In the next Section, we will elaborate on this property, and briefly consider two simple examples, before illustrating the application to Fibonacci numbers in Section \ref{S_FibFamily}.


\section{Two Simple Examples of Families of Number Sequences}
\label{S_Examples}


\subsection{The Family of Power Sequences}

Let $x_{n,l} = c = const \in \R$. In this case, we have
\begin{eqnarray}
X_{n,m} & = & \prod\limits_{l=1}^n ( m + c ) \equiv (m+c)^n \label{Eq_FPowerX} \\
\mathscr{X}_n & = & \sum\limits_{l=1}^n c \equiv nc \label{Eq_FPowerCalX} \, .
\end{eqnarray}
The first few members of this family, for $c=0$, are listed in Table \ref{Tab_FPower}. With the results presented in the previous Section, we can immediately formulate

\begin{table}
\caption{\label{Tab_FPower} The first members of the family of power sequences $X_{n,m} = m^n, n \in \N$ and $m \in \Z$, equation (\ref{Eq_FPowerX}) with $c=0$.}
\begin{tabular}{|c|cccccccccc|}
\hline
\diagbox{n}{m} & \ldots & 0 & 1 & 2 & 3 & 4 & 5 & 6 & 7 & \ldots \\
\hline
1 & & 0 & 1 & 2 & 3 & 4 & 5 & 6 & 7 & \\
2 & & 0 & 1 & 4 & 9 & 16 & 25 & 36 & 49 & \\
3 & & 0 & 1 & 8 & 27 & 64 & 125 & 216 & 343 & \\
4 & & 0 & 1 & 16 & 81 & 256 & 625 & 1296 & 2401 & \\
5 & & 0 & 1 & 32 & 243 & 1024 & 3125 & 7776 & 16807 & \\
6 & & 0 & 1 & 64 & 729 & 4096 & 15625 & 46656 & 117649 & \\
7 & & 0 & 1 & 128 & 2187 & 16384 & 78125 & 279936 & 823543 & \\
\vdots & & & & & & & & & & \\
\hline
\end{tabular}
\end{table}

\begin{corollary}
\label{Cor_FPower}
The family of power sequences $X_{n,m} = (m+c)^n$ obeys for all $c \in \R$
\begin{eqnarray*}
\sum\limits_{l=1}^n (-1)^l \binom{n}{l} l (l+m+c)^n & = & \frac{1}{2} (-1)^n (2m+2c+n+1) n n! \\
\sum\limits_{l=1}^n (-1)^l \binom{n}{l} l (lm+c)^n & = & \frac{1}{2} (-1)^n m^{n-1} \left( 2c+m(n+1) \right) n n! \\
\sum\limits_{l=0}^n (-1)^l \binom{n}{l} (m+c+1-l)^n & = & n! \\
\sum\limits_{l=1}^n (-1)^l \binom{n}{l} l \left( (lm+c)^n - m^{n-1} (l+c)^n \right) & = & \frac{1}{2} (-1)^{n-1} m^{n-1} (1-m) n(n+1)! 
\end{eqnarray*}
for all $n \in \N$ and $m \in \Z$, 
\begin{eqnarray*}
\sum\limits_{l=0}^{n-1} (-1)^l \binom{n}{l} \frac{n-l}{l-m} (c+l)^n & = & (-1)^n \left( \frac{(m-n)!}{m!} (c+m)^n - 1 \right) n! \\
\sum\limits_{l=0}^{n-1} (-1)^l \binom{n}{l} \frac{n-l}{l-m} (c-l)^n & = & (-1)^n \left( \frac{(m-n)!}{m!} (c-m)^n - (-1)^n \right) n! 
\end{eqnarray*}
for all $n, m \in \N$ with $m \geq n$, and 
\begin{eqnarray*}
\sum\limits_{l=0}^n (-1)^l \binom{n}{l} (n-l)^q (m+c-l)^{n-p} & = & 0 \qquad \text{for } 0 \leq q < p \\
\sum\limits_{l=0}^n (-1)^l \binom{n}{l} (n-l)^p (m+c-l)^{n-p} & = & n!
\end{eqnarray*}
for all $n,p \in \N$ with $n \geq p+1, p \geq 1$ and $m \in \Z$. 
\end{corollary}

\begin{proof}
All identities in Corollary \ref{Cor_FPower} are a direct consequence of Lemmata \ref{L_calX} and \ref{L_calXm}, Proposition \ref{P_XnmIdent} and Corollaries \ref{Cor_Xnm1} and \ref{Cor_Xnm2}, using (\ref{Eq_FPowerX}) and (\ref{Eq_FPowerCalX}).
\end{proof}

We note that Corollary \ref{Cor_FPower} yields a number of interesting combinatorial, in particular binomial, identities and their generalizations. Specifically, for $c=0$, Gould (1.13), (1.14) and (1.47) are recovered \cite{Gould72}. Furthermore, for any fixed $n$, $X_{n,m}$ yields the sequence of $n$th powers of subsequent integers $m$. The third relation in Corollary \ref{Cor_FPower}, for $c=0$, provides then the general form of the $n$th-order linear homogeneous recursions in $m$ with constant coefficients for such sequences:
\begin{equation}
\label{Eq_FPowerRecM}
(m+1)^n = \sum\limits_{l=0}^{n-1} (-1)^l \binom{n}{l+1} (m-l)^n + n!
\end{equation}
$\forall m \in \Z$. For example, restricting to $m \geq 0$, we obtain for $n=2$ the sequence of square numbers $a_m = m^2$ \cite{OEIS1}, obeying the known linear recursion
\begin{equation*}
a_0 = 0, a_1 = 1, a_{m+1} = 2 a_m - a_{m-1} + 2, m \geq 1
\end{equation*}
(see M. Kristof, 2005, \cite{OEIS1}), for $n=3$ the sequence of cubes $a_m = m^3$ \cite{OEIS2}, obeying
\begin{equation*}
a_0 = 0, a_1 = 1, a_2 = 2^3, a_{m+1} = 3 a_m - 3 a_{m-1} + a_{m-2} + 6, m \geq 2
\end{equation*}
(see A. King, 2013, \cite{OEIS2}), and for $n=4$ the sequence $a_m = m^4$ \cite{OEIS3}, subject to the $4$th-order linear recursion
\begin{equation*}
a_0 = 0, a_1 = 1, a_2 = 2^4, a_3 = 3^4, a_{m+1} = 4 a_m - 6 a_{m-1} + 4 a_{m-2} - a_{m-3} + 24, m \geq 3
\end{equation*}
(see A. King, 2013, \cite{OEIS3}). 


\subsection{The Family of Pochhammer Sequences}

Let $x_{n,l} = l \in \N$. In this case, we have
\begin{eqnarray}
X_{n,m} & = & \prod\limits_{l=1}^n ( m + l ) \equiv (m+1)_n \label{Eq_FPochhammerX} \\
\mathscr{X}_n & = & \sum\limits_{l=1}^n l \equiv \frac{1}{2} n (n+1) \label{Eq_FPochhammerCalX} \, ,
\end{eqnarray}
where $(a)_n = \Gamma[a+n] / \Gamma[a]$ denotes the Pochhammer symbol. The first members of this family of sequences is visualized in Table \ref{Tab_FPochhammer}. With the results presented in the last Section, we can immediately formulate

\begin{corollary}
\label{Cor_FPochhammer}
The family of Pochhammer sequences $X_{n,m} = (m+1)_n$ obeys
\begin{eqnarray*}
\sum\limits_{l=1}^n (-1)^l \binom{n}{l} l (l+m)_n & = & (-1)^n (m+n) n n! \\
\sum\limits_{l=1}^n (-1)^l \binom{n}{l} (lm)_{n+1} & = & \frac{1}{2} (-1)^n m^n (1+m) n (n+1)! \\
\sum\limits_{l=0}^n (-1)^l \binom{n}{l} (l+m-n+1)_n & = & (-1)^n n! \\
\sum\limits_{l=1}^n (-1)^l \binom{n}{l} \left( (lm)_{n+1} - m^n (l)_{n+1} \right) & = & \frac{1}{2} (-1)^{n+1} m^n (1-m) n (n+1)! \\
\end{eqnarray*}
for all $n \in \N$ and $m \in \Z$, 
\begin{eqnarray*}
\sum\limits_{l=0}^{n-1} (-1)^l \binom{n}{l} \frac{n-l}{l-m} (1+l)_n & = & (-1)^n \left( \frac{(m-n)!}{m!} (1+m)_n - 1 \right) n! \\
\sum\limits_{l=0}^{n-1} (-1)^l \binom{n}{l} \frac{n-l}{l-m} (1-l)_n & = & (-1)^n \left( \frac{(m-n)!}{m!} (1-m)_n - (-1)^n \right) n! 
\end{eqnarray*}
for all $n, m \in \N$ with $m \geq n$, and 
\begin{eqnarray*}
\sum\limits_{l=0}^n (-1)^l \binom{n}{l} l^q (m-n+l+1)_{n-p} & = & 0 \qquad \text{for } 0 \leq q < p \\
\sum\limits_{l=0}^n (-1)^l \binom{n}{l} l^p (m-n+l+1)_{n-p} & = & n!
\end{eqnarray*}
for all $n,p \in \N$ with $n \geq p+1, p \geq 1$ and $m \in \Z$. 
\end{corollary}

\begin{proof}
All identities in Corollary \ref{Cor_FPochhammer} are a direct consequence of Lemmata \ref{L_calX} and \ref{L_calXm}, Proposition \ref{P_XnmIdent} and Corollaries \ref{Cor_Xnm1} and \ref{Cor_Xnm2}, using (\ref{Eq_FPowerX}) and (\ref{Eq_FPowerCalX}) and the Pochhammer identity $l (lm+1)_n = (lm)_{n+1} / m$ valid $\forall m \in \Z$.
\end{proof}

\begin{table}
\caption{\label{Tab_FPochhammer} The first members of the family of Pochhammer sequences $X_{n,m} = (m+1)_n, n \in \N, m \in \Z$, equation (\ref{Eq_FPochhammerX}).}
\begin{tabular}{|c|cccccccccc|}
\hline
\diagbox{n}{m} & \ldots & 0 & 1 & 2 & 3 & 4 & 5 & 6 & 7 & \ldots \\
\hline
1 & & 1 & 2 & 3 & 4 & 5 & 6 & 7 & 8 & \\
2 & & 2 & 6 & 12 & 20 & 30 & 42 & 56 & 72 & \\
3 & & 6 & 24 & 60 & 120 & 210 & 336 & 504 & 720 & \\
4 & & 24 & 120 & 360 & 840 & 1680 & 3024 & 5040 & 7920 & \\
5 & & 120 & 720 & 2520 & 6720 & 15120 & 30240 & 55440 & 95040 & \\
6 & & 720 & 5040 & 20160 & 60480 & 151200 & 332640 & 665280 & 1235520 & \\
7 & & 5040 & 40320 & 181440 & 604800 & 1663200 & 3991680 & 8648640 & 17297280 & \\
\vdots & & & & & & & & & & \\
\hline
\end{tabular}
\end{table}

As in the case of the family of power sequences, for any given $n$, the relations listed in Corollary \ref{Cor_FPochhammer} provide links between Pochhammer numbers $(m)_n$ for different $m$. Specifically, the third identity yields the general linear recursive rule for sequences defined by $(m+1)_n$ for any fixed $n$:
\begin{equation}
\label{Eq_FPochhammerRecM}
(m+1)_n = \sum\limits_{l=0}^{n-1} (-1)^l \binom{n}{l+1} (m-l)_n + n!
\end{equation}
valid $\forall m \in \Z$. Restricting again to $m \geq 0$, $n=2$ yields the sequence of Oblong numbers $a_m = m(m+1)$ \cite{OEIS4}, subject to the recursion
\begin{equation*}
a_0 = 0, a_1 = 2, a_{m+1} = 2 a_m - a_{m-1} + 2, m \geq 1,
\end{equation*}
for $n=3$ we obtain the sequence $a_m = m(m+1)(m+2)$ \cite{OEIS5} obeying
\begin{equation*}
a_0 = 0, a_1 = 3!, a_2 = 4!, a_{m+1} = 3 a_m - 3 a_{m-1} + a_{m-2} + 6, m \geq 2
\end{equation*}
(see Z. Seidov, 2006, \cite{OEIS5}), and for $n=4$ the sequence of products of four consecutive integers $a_m = m(m+1)(m+2)(m+3)$  \cite{OEIS6} with
\begin{equation*}
a_0 = 0, a_1 = 4!, a_2 = 5!, a_3 = \frac{1}{2} 6!, a_{m+1} = 4 a_m - 6 a_{m-1} + 4 a_{m-2} - a_{m-3} + 24, m \geq 3 \, .
\end{equation*}
We note that already for $n=2$ and $n=4$, the recursions obtained here differ in form from those provided in OEIS \cite{OEIS} for the corresponding sequences. 


\section{The Family of $k$-generalized Fibonacci Numbers}
\label{S_FibFamily}

In the remainder of this contribution, we will apply the general results presented in Section \ref{S_ProdRep} to a less trivial case, namely the generalized Fibonacci sequences defined in (\ref{Eq_GenFibRec}). To that end, we set
\begin{equation}
x_{n,l} = - 2 i \cos\left[ \frac{l \pi}{n+1} \right] ,
\end{equation}
from which, using (\ref{Eq_GenFibProd}), immediately follows that
\begin{equation}
\label{Eq_FFibX}
X_{n,m} = \prod\limits_{l=1}^n \left( m - 2 i \cos\left[ \frac{l \pi}{n+1} \right] \right) \equiv F_{n+1}^{(m)} \, .\end{equation}
Furthermore, noting that $\cos \left[ \frac{l \pi}{n+1} \right], l \in [0,n]$ are the zeros of Chebyshev polynomials of the second kind $U_n(x)$, we have 
\begin{equation}
\label{Eq_FFibCalX}
\mathscr{X}_n = -2 i \sum\limits_{l=1}^n \cos\left[ \frac{l \pi}{n+1} \right] \equiv 0 \, ,
\end{equation}
where the orthogonality relations for Chebyshev polynomials were used. The first members of the family of sequences formed by (\ref{Eq_FFibX}) are visualized in Table \ref{Tab_FFibonacci}. Specifically, the second column $m=1$ contains the original Fibonacci sequence $F_n$, equation (\ref{Eq_Fibs}), and the third column $m=2$ the sequence of Pell numbers $P_n$, equation (\ref{Eq_Pells}), for $n \geq 1$. 

While individual columns yield subsequent sequences of generalized Fibonacci numbers, each row, for fixed $n$, generates new integer sequences whose elements are all generalized Fibonacci numbers. Specifically, for $n=2$, we obtain, for $m \geq 0$, the sequence $a_m = m^2 + 1$ \cite{OEIS7}, for $n=3$ the sequence $a_m = m^3 + 2m$ \cite{OEIS8} and for $n=4$ the sequence $a_m = m^4 + 3m^2 + 1$ \cite{OEIS9}. In general, for any given $n \in \N$, integer sequences are obtained whose explicit form is given in terms of Fibonacci polynomials (e.g., see \cite{AndreJeannin94}), polynomials of $n$th order in the sequence index $m$ of the form
\begin{equation}
\label{Eq_mFibRec}
a_m = \sum\limits_{l=0}^{\lfloor \frac{n}{2} \rfloor} \binom{n-l}{l} m^{n-2l} \equiv F_{n+1}^{(m)} \, .
\end{equation}
From Lemmata \ref{L_calX} and \ref{L_calXm}, we can immediately formulate

\begin{proposition}
\label{Cor_FFib1}
The family of generalized Fibonacci sequences $\{ F_n^{(m)} \}$, defined in  (\ref{Eq_FFibX}), obeys the following identities:
\begin{eqnarray}
\sum\limits_{l=1}^n (-1)^l \binom{n}{l} l F_{n+1}^{(l+m)} & = & \frac{1}{2} (-1)^n (2m+n+1) n n! \label{Eq_FFibCor1_1} \\
\sum\limits_{l=1}^n (-1)^l \binom{n}{l} l F_{n+1}^{(lm)} & = & \frac{1}{2} (-1)^n m^n n (n+1)! \label{Eq_FFibCor1_2}
\end{eqnarray}
for all $n \in \N$ and $m \in \Z$. 
\end{proposition}

\begin{proof}
Both identities are a direct consequence of (\ref{Eq_calXm1}) and (\ref{Eq_calXm2}), using (\ref{Eq_FFibX}) and (\ref{Eq_FFibCalX}).
\end{proof}

\begin{table}
\caption{\label{Tab_FFibonacci} The first members of the family of generalized Fibonacci numbers $X_{n,m} = F_{n+1}^{(m)}, n \in \N, m \in \Z$, defined explicitly in (\ref{Eq_FFibX}) and subject to the recursive relation (\ref{Eq_GenFibRec}).}
\begin{tabular}{|c|cccccccccc|}
\hline
\diagbox{n}{m} & \ldots & 0 & 1 & 2 & 3 & 4 & 5 & 6 & 7 & \ldots \\
\hline
1 & & 0 & 1 & 2 & 3 & 4 & 5 & 6 & 7 & \\
2 & & 1 & 2 & 5 & 10 & 17 & 26 & 37 & 50 & \\
3 & & 0 & 3 & 12 & 33 & 72 & 135 & 228 & 357 & \\
4 & & 1 & 5 & 29 & 109 & 305 & 701 & 1405 & 2549 & \\
5 & & 0 & 8 & 70 & 360 & 1292 & 3640 & 8658 & 18200 & \\
6 & & 1 & 13 & 169 & 1189 & 5473 & 18901 & 53353 & 129949 & \\
7 & & 0 & 21 & 408 & 3927 & 23184 & 98145 & 328776 & 927843 & \\
\vdots & & & & & & & & & & \\
\hline
\end{tabular}
\end{table}

Equation (\ref{Eq_FFibCor1_2}) can be used to link generalized Fibonacci numbers $F_n^{(m)}$ for positive and negative $m$. Specifically, subtracting (\ref{Eq_FFibCor1_2}) for a given $m$ and $m \rightarrow -m$, we obtain
\begin{equation}
\sum\limits_{l=1}^n (-1)^l \binom{n}{l} l \left( F_{n+1}^{(lm)} - F_{n+1}^{(-lm)} \right)  = \frac{1}{2} m^n \left( (-1)^n - 1 \right) n (n+1)! 
\end{equation}
which, for $m=1$, yields
\begin{equation}
\label{Eq_FibPosNeg1}
\sum\limits_{l=1}^n (-1)^l \binom{n}{l} l \left( F_{n+1}^{(-l)} - F_{n+1}^{(l)} \right)  = \frac{1}{2} \left( 1 - (-1)^n \right) n (n+1)! 
= \left\{ \begin{array}{ll} 
0 & n \text{ even} \\ 
n(n+1)! & n \text{ odd.}
\end{array} \right.
\end{equation}
Furthermore, application of Proposition \ref{P_XnmIdent} to the family of Fibonacci sequences leads to

\begin{proposition}
\label{Cor_FFib2}
The family of Fibonacci sequences $F_n^{(m)}$ obeys $\forall n \in \N$ the following relations
\begin{eqnarray}
\sum\limits_{l=0}^n (-1)^l \binom{n}{l} F_{n+1}^{(l+m-n+1)} & = & (-1)^n n! \label{Eq_FFibCor2_1} \\
\sum\limits_{l=1}^n (-1)^l \binom{n}{l} l \left( F_{n+1}^{(lm)} - m^{n-1} F_{n+1}^{(l)} \right) & = & \frac{1}{2} (-1)^{n-1} m^{n-1} (1-m) n (n+1)! \label{Eq_FFibCor2_2}
\end{eqnarray}
for all $m \in \Z$. 
\end{proposition}

\begin{proof}
Both identities are a consequence of (\ref{Eq_XnmRec}) and (\ref{Eq_XnmNeg1}), using (\ref{Eq_FFibX}) and (\ref{Eq_FFibCalX}). Relation (\ref{Eq_FFibCor2_2}) can be directly obtained also from (\ref{Eq_FFibCor1_2}).
\end{proof}

We note that equation (\ref{Eq_FFibCor2_2}) is a special application of (\ref{Eq_FFibCor1_2}), which for $m=-1$ yields
\begin{equation}
\sum\limits_{l=1}^n (-1)^l \binom{n}{l} l \left( F_{n+1}^{(-l)} + (-1)^n F_{n+1}^{(l)} \right) = n (n+1)! 
\end{equation}
$\forall n \in \N$, complementing (\ref{Eq_FibPosNeg1}) above. Interestingly, (\ref{Eq_FFibCor2_1}) allows to construct, for any given $n$, general recursive relations in $m$ for generalized Fibonacci numbers $F_n^{(m)}$:
\begin{equation}
\label{Eq_GenFibRecM}
F_n^{(m+1)} = \sum\limits_{l=0}^{n-1} (-1)^l \binom{n}{l+1} F_n^{(m-l)} + n!
\end{equation}
valid $\forall m \in \Z$. Restricting to $m \geq 0$, the resulting sequence $a_m \equiv F_2^{(m+1)} = m^2 + 1$ \cite{OEIS7} obtained for $n=2$ from (\ref{Eq_FFibCor2_1}), obeys
\begin{equation*}
a_0 = 1, a_1 = 2, a_{m+1} = 2 a_m - a_{m-1} + 2, m \geq 1
\end{equation*}
(see E. Werley, 2011, \cite{OEIS7}). Similarly, for $n=3$, the integer sequence given by the third-order polynomial $a_m \equiv F_3^{(m+1)} = m^3 + 2m$ \cite{OEIS8}, obeys the recursive relation
\begin{equation*}
a_0 = 0, a_1 = 3, a_2 = 12, a_{m+1} = 3 a_m - 3 a_{m-1} + a_{m-2} + 6, m \geq 2 \, ,
\end{equation*}
and for $n=4$, the sequence $a_m \equiv F_4^{(m+1)} = m^4 + 3m^2 + 1$ \cite{OEIS9} is subject to the recursion
\begin{equation*}
a_0 = 1, a_1 = 5, a_2 = 29, a_2 = 109, a_{m+1} = 4a_m - 6 a_{m-1} + 4a_{m-2} - a_{m-3} + 24, m \geq 3 \, .
\end{equation*}

Finally, Corollaries \ref{Cor_Xnm1} and \ref{Cor_Xnm2} provide explicit representations of generalized Fibonacci numbers in terms of other members of the Fibonacci family:

\begin{proposition}
\label{Cor_FFib3}
Generalized Fibonacci numbers $F_n^{(m)}$ obey
\begin{eqnarray}
F_{n+1}^{(m)} & = & (-1)^n \binom{m}{n} \sum\limits_{l=0}^{n-1} (-1)^l \binom{n}{l} \frac{n-l}{l-m} F_{n+1}^{(l)} + \frac{m!}{(m-n)!} \label{Eq_FFibCor3_1} \\
F_{n+1}^{(-m)} & = & (-1)^n \binom{m}{n} \sum\limits_{l=0}^{n-1} (-1)^l \binom{n}{l} \frac{n-l}{l-m} F_{n+1}^{(-l)} + (-1)^n \frac{m!}{(m-n)!} \label{Eq_FFibCor3_2} 
\end{eqnarray}
$\forall n, m \in \N$ with $m \geq n$, and
\begin{eqnarray}
F_{n-p+1}^{(m)} & = & (-1)^{n+1} n^{-q} \sum\limits_{l=0}^{n-1} (-1)^l \binom{n}{l} l^q F_{n-p+1}^{(m-n+l)} \label{Eq_FFibCor3_3} \\
F_{n-p+1}^{(m)} & = & (-1)^{n+1} n^{-p} \sum\limits_{l=0}^{n-1} (-1)^l \binom{n}{l} l^p F_{n-p+1}^{(m-n+l)} + n^{-p} n!\label{Eq_FFibCor3_4}
\end{eqnarray}
for all $n,p \in \N$ with $n \geq p+1, p \geq 1$, $0 \leq q < p$ and $m \in \Z$.
\end{proposition}

\begin{proof}
The first two identities are a direct consequence of (\ref{Eq_XnmExp1}) and (\ref{Eq_XnmExp1}), the last two can be shown with (\ref{Eq_XnmId1}) and (\ref{Eq_XnmId2}), using (\ref{Eq_FFibX}) and (\ref{Eq_FFibCalX}).
\end{proof}

We note that Propositions \ref{Cor_FFib1} to \ref{Cor_FFib3} provide a number of identities which interlink the set of generalized Fibonacci sequences defined in (\ref{Eq_FFibX}). Specifically, the defining recursive relations in $n$, equation (\ref{Eq_GenFibRec}), for generalized Fibonacci numbers allow to express each $F_n^{(m)}$ in terms of $F_{n'}^{(m)}$, $n' < n$ for fixed $m$, whereas relations (\ref{Eq_FFibCor1_1})--(\ref{Eq_FFibCor1_2}), (\ref{Eq_FFibCor2_1})--(\ref{Eq_FFibCor2_2}) and (\ref{Eq_FFibCor3_1})--(\ref{Eq_FFibCor3_4}) allow to express each $F_n^{(m)}$ in terms of $F_n^{(m')}$, $m' \neq m$ for any given $n$. Combining both sets of identities, we arrive at linking all members of the family of generalized Fibonacci numbers.


\section{Concluding Remarks}

In this contribution, we investigated general properties of number sequences generated through products of the form (\ref{Eq_Xnm}). We found several identities and recursive relations which interlink such sequences and suggest their classification in terms of families (Definition \ref{Def_Family}). Although the families studied here as examples describe different integer sequences, such as Pochhammer numbers, powers of integers or generalized Fibonacci numbers, we find that each of these families is subject to the same set of identities which, in some cases, generalize interesting relations between these known sequences. 

The examples presented here constitute but a small set of potential applications. For instance, $q$-Poch\-hammer sequences and sequences produced by products of Pochhammer numbers are obtained for $x_{n,l} = a^l, a \in \R$ and $x_{n,l} = l^a, a \in \Z$, respectively. The general relations listed in Section \ref{S_ProdRep} apply in these cases, and provide a number of identities obeyed by the corresponding sequences. By setting 
\begin{equation*}
x_{n,l} = -2 \sqrt{q} \cos \left[ \frac{l \pi}{n+1} \right]
\end{equation*}
we obtain, with (\ref{Eq_ProdRepLucas}), general Lucas sequences \cite{OEIS10}, i.e.
\begin{equation*}
X_{n,m} = \frac{1}{\sqrt{q}} \prod\limits_{l=1}^n \left( m - 2 \sqrt{q} \cos \left[ \frac{l \pi}{n+1} \right] \right) \equiv L_{n+1}^{(m,q)} \, .
\end{equation*}
For appropriate $m$ and $q$, interesting identities, such as grandma's identity \cite{Humble04}, signed bisections of Fibonacci sequences and relations interlinking powers of Fibonacci numbers, are obtained. The study of these relations, their potential generalization and application to other families of number sequences might provide novel and potentially useful insights into properties shared by qualitatively different number sequences.


\medskip

\noindent MSC2010: 11B39, 32A05, 40B05
\end{document}